\documentclass[12pt, reqno]{amsart}
\usepackage{amsmath, amsthm, amscd, amsfonts, amssymb, graphicx, color}
\usepackage[bookmarksnumbered, colorlinks, plainpages]{hyperref}
\hypersetup{colorlinks=true,linkcolor=red, anchorcolor=green, citecolor=cyan, urlcolor=red, filecolor=magenta, pdftoolbar=true}

\newtheorem{theorem}{Theorem}[section]
\newtheorem{lemma}[theorem]{Lemma}

\newtheorem{cor}[theorem]{Corollary}
\theoremstyle{definition}

\newtheorem{example}[theorem]{Example}

\theoremstyle{remark}
\newtheorem{remark}[theorem]{\bf{Remark}}
\numberwithin{equation}{section}
\begin{document}

\title [Refined inequalities for the numerical radius  of  operators]{Refined inequalities for the numerical radius of Hilbert space operators}

\author[P. Bhunia, S. Jana and K. Paul]{Pintu Bhunia, Suvendu Jana and Kallol Paul}

\address [Bhunia] {Department of Mathematics, Jadavpur University, Kolkata 700032, West Bengal, India}
\email{pintubhunia5206@gmail.com}

\address[Jana] {Department of Mathematics, Mahisadal Girls College, Purba Medinipur 721628, West Bengal, India}
\email{janasuva8@gmail.com}

\address[Paul] {Department of Mathematics, Jadavpur University, Kolkata 700032, West Bengal, India}
\email{kalloldada@gmail.com}




\maketitle
\begin{abstract}
We present some new upper and lower bounds for the numerical radius of bounded linear operators on a complex Hilbert space and show that these are stronger than the existing ones. In particular, we prove that if $A$ is a bounded linear operator on a complex Hilbert space $\mathcal{H}$ and if $\Re(A)$, $\Im(A)$ are the real part, the imaginary part of $A$, respectively, then
$$ w(A)\geq\frac{\|A\|}{2} +\frac{1}{2\sqrt{2}}  \Big | \|\Re(A)+\Im(A)\|-\|\Re(A)-\Im(A)\|  \Big | $$ 
and
$$ w^2(A)\geq\frac{1}{4}\|A^*A+AA^*\|+\frac{1}{4}\Big| \|\Re(A)+\Im(A)\|^2-\|\Re(A)-\Im(A)\|^2\Big|. $$
Here $w(.)$ and $\|.\|$ denote the numerical radius and the operator norm, respectively.  Further, we obtain refinement of  inequalities for the numerical radius of the product of two operators.  Finally, as an application of the second inequality mentioned above, we obtain an  improvement of upper bound for the numerical radius of the commutators of operators.

\end{abstract}

\section{Introduction}

\noindent Let $ \mathcal{B}(\mathcal{H})$ denote the $\mathcal{C}^*$-algebra of all bounded linear operators on a complex Hilbert space $\mathcal{H}$ with inner product $\langle .,.  \rangle $ and the corresponding norm $\|.\|$  induced by the inner product $\langle .,. \rangle .$  For $A\in \mathcal{B}(\mathcal{H})$, let $A^*$ be the adjoint of $A$ and $|A|=({A^*A})^{\frac{1}{2}}$. Let $\Re(A)$ and $\Im(A)$ denote the real part and the imaginary part of $A$, respectively, i.e., $\Re(A)=\frac{1}{2}(A+A^*)$ and $\Im(A)=\frac{1}{2\rm i}(A-A^*)$. 
Let $\|A\|$ and $w(A)$ denote the operator norm and the numerical radius of $A$, respectively. Recall that $$w(A)=\sup\{|\langle Ax,x  \rangle|: x\in \mathcal{H}, \|x\|=1\}.$$ 
It is well known that the numerical radius, i.e., $ w(.)$ defines a norm on $\mathcal{B}(\mathcal{H})$ and is equivalent to the operator norm $\|.\|$. In fact, for every $ {A}\in\mathcal{B}(\mathcal{H})$,  we have
\begin{eqnarray}\label{eqv}
\frac{1}{2} \|A\|\leq w({A})\leq\|A\|.
\end{eqnarray}
The inequalities in (\ref{eqv}) are sharp, $w(A)=\frac{\|A\|}{2}$ if $A^2=0$ and $w(A)=\|A\|$ if $A^*A=AA^*.$ 
In \cite{E}, Kittaneh improved on the inequalities in  (\ref{eqv}) and proved that 
\begin{eqnarray}
\frac{1}{4}\|A^*A+A{A}^*\|\leq w^2({A})\leq\frac{1}{2}\|A^*A+A{A}^*\|.
\label{d}\end{eqnarray}
Also in \cite{M}, Kittaneh refined the second inequality in (\ref{eqv}) to prove that 
\begin{eqnarray}\label{b}
w(A)\leq\frac{1}{2}(\|A\|+\|A^2\|^{\frac{1}{2}}).
\end{eqnarray}
The Aluthge transform of ${A}$, denoted as $\tilde{{A}}$, is defined as $\tilde{{A}}=|{A}|^\frac{1}{2}U|{A}|^\frac{1}{2}$,
where $U$ is the (unique) partial isometry appearing in the polar
decomposition $ A =U|A|$ of $A$ with $ \ker A = \ker U.$ It follows from the
definition of $\tilde{A}$ that $ \|\tilde{A}\|\leq\|A\|$. Also, $ w(\tilde{A})\leq w(A)$. In \cite{D}, Yamazaki proved that if $A\in B(\mathcal{H}) $, then 
\begin{eqnarray}\label{e}
 w(A)\leq\frac{1}{2}(\|A\|+w(\tilde{A})).
 \end{eqnarray}
He also proved that (\ref{e}) refine (\ref{b}). For more recent  results on the refinement of the numerical radius inequalities of bounded linear operators we refer the readers to see \cite{P7,P18,P8,P10,P1}.
Next we note some well known inequalities for the numerical radius of product of two operators. 
Dragomir  \cite{aa}  proved  that if A,B $\in\mathcal{B}(\mathcal{H})$  and $r\geq1$, then
 \begin{eqnarray}\label{f}
 w^r(B^*A)\leq\frac{1}{2}(\||A|^{2r}+|B|^{2r}\|).
 \end{eqnarray}
 In \cite{AA}, Hedarbeygi et al. established a refinement of (\ref{f}) and proved that if $ A,B \in\mathcal{B}(\mathcal{H})$ and $r\geq1$, then
  \begin{eqnarray}\label{h}
   w^{2r}(B^*A)\leq\frac{1}{2}w^r(|B|^2|A|^2)+\frac{1}{4}\| |B|^{4r}+|A|^{4r}\|.
\end{eqnarray} 
 \smallskip
In section 2, we obtain several upper and lower bounds for the numerical radius of bounded linear operators on $\mathcal{H}$, which improve on the existing ones in (\ref{eqv}), (\ref{d}), (\ref{b}) and (\ref{e}). In section 3, we obtain upper bounds for the numerical radius of the product of two operators, which refine (\ref{f}) and (\ref{h}). As a direct application of the results  obtained here,  we obtain an  improvement of upper bound for the numerical radius of the commutators of operators.

\section{Numerical radius inequalities of operators}

We start our work with the following improvement of the first inequality in (\ref{eqv}).

\begin{theorem}\label{thn}
If $ A\in\mathcal{B}(\mathcal{H})$, then 
$$ w(A)\geq\frac{1}{2} \|A\|+\frac{1}{2\sqrt{2}}  \mid \|\Re(A)+\Im(A)\|-\|\Re(A)-\Im(A)\|  \mid.  $$ 
\end{theorem}
\begin{proof}
 Let $x \in \mathcal{H}$ with $\|x\|=1$. Then by Cartesian decomposition of $A$, we have
  \begin{eqnarray*}
  	\mid\langle Ax,x\rangle\mid &=& \sqrt{\langle \Re(A)x,x\rangle^2+\langle \Im(A)x,x\rangle^2}\\
  	&\geq& \frac{1}{\sqrt{2}}(\mid\langle \Re(A)x,x\rangle\mid+\mid\langle \Im(A)x,x\rangle\mid)\\ &\geq&\frac{1}{\sqrt{2}}\mid\langle (\Re(A)\pm  \Im(A))x,x\rangle\mid.
  \end{eqnarray*}
Taking supremum over all $x\in \mathcal{H}$,  $\|x\|=1$, we get 
$$ w(A)\geq\frac{1}{\sqrt{2}}\|\Re(A)\pm \Im(A)\|. $$ 
Thus, we have $$w(A)\geq \frac{1}{\sqrt{2}}\max\{\|\Re(A)+\Im(A)\|,\|\Re(A)-\Im(A)\|\}.$$
Now,
\begin{eqnarray*}
	&&\frac{1}{\sqrt{2}}\max\{\|\Re(A)+\Im(A)\|,\|\Re(A)-\Im(A)\|\}\\
	&=&\frac{1}{\sqrt{2}}\left\lbrace\frac{ \|\Re(A)+\Im(A)\|+\|\Re(A)-\Im(A)\|}{2}+\frac{\mid \|\Re(A)+\Im(A)\|-\|\Re(A)-\Im(A)\|\mid}{2}\right\rbrace\\ &\geq&\frac{1}{\sqrt{2}}\left\lbrace\frac{ \|(\Re(A)+\Im(A))+\rm i(\Re(A)-\Im(A))\|}{2}+\frac{\mid \|\Re(A)+\Im(A)\|-\|\Re(A)-\Im(A)\|\mid}{2}\right\rbrace\\
	&=&\frac{1}{\sqrt{2}}\left\lbrace\frac{ \|(1+\rm i)A^*\|}{2}+\frac{\mid \|\Re(A)+\Im(A)\|-\|\Re(A)-\Im(A)\|\mid}{2}\right\rbrace\\
	 &=&\frac{\|A\|}{2}+\frac{\mid \|\Re(A)+\Im(A)\|-\|\Re(A)-\Im(A)\|\mid}{2\sqrt{2}}.
\end{eqnarray*} 
This completes the proof.
\end{proof}

\begin{remark}\label{remark1}
(i) Clearly, the inequality in Theorem \ref{thn} refines the first inequality in (\ref{eqv}).\\
(ii)  If $ w(A)=\frac{\|A\|}{2}, $ then $ \|\Re(A)+\Im(A)\|=\|\Re(A)-\Im(A)\|,$ but the converse is not necessarily true.\\
(iii) In \cite[Th. 2.1]{PK21}, Bhunia and Paul proved that 
\begin{eqnarray}\label{pk1}
	w(A)&\geq &\frac{1}{2} \left \|A \right\| +  \frac{  1}{2} \mid \|\Re(A)\|-\|\Im(A)\|\mid.
\end{eqnarray}
Clearly, if $A\in \mathcal{B}(\mathcal{H})$ is (non-zero) self-adjoint, then 
$$\frac{1}{2} \left \|A \right\| +  \frac{  1}{2} \mid \|\Re(A)\|-\|\Im(A)\|\mid> \frac{\|A\|}{2}+\frac{\mid \|\Re(A)+\Im(A)\|-\|\Re(A)-\Im(A)\|\mid}{2\sqrt{2}}$$ and if $A=\begin{pmatrix}
1+{\rm i} & 0\\
0 &0
\end{pmatrix} $, then 
$$\frac{1}{2} \left \|A \right\| +  \frac{  1}{2} \mid \|\Re(A)\|-\|\Im(A)\|\mid< \frac{\|A\|}{2}+\frac{\mid \|\Re(A)+\Im(A)\|-\|\Re(A)-\Im(A)\|\mid}{2\sqrt{2}}.$$
Thus, we conclude that the inequality in Theorem \ref{thn} and the inequality (\ref{pk1}) are not comparable, in general.
\end{remark}

Next we prove the following inequality which improve on the first inequality in (\ref{d}). 

\begin{theorem}
If $ A\in\mathcal{B}(\mathcal{H})$, then 
$$ w^2(A)\geq\frac{1}{4}\|A^*A+AA^*\|+\frac{1}{4}\mid \|\Re(A)+\Im(A)\|^2-\|\Re(A)-\Im(A)\|^2\mid. $$
\label{thnn}\end{theorem}

\begin{proof}
As in the proof of Theorem \ref{thn}, we have
 $$w^2(A)\geq\frac{1}{2}\max\{\|\Re(A)+\Im(A)\|^2,\|\Re(A)-\Im(A)\|^2\}.$$
 Now,
 \begin{eqnarray*}
 &&	\frac{1}{2}\max\{\|\Re(A)+\Im(A)\|^2,\|\Re(A)-\Im(A)\|^2\}\\
 	&=& \frac{1}{2}\left\lbrace\frac{ \|\Re(A)+\Im(A)\|^2+\|\Re(A)-\Im(A)\|^2}{2}+\frac{\mid \|\Re(A)+\Im(A)\|^2-\|\Re(A)-\Im(A)\|^2\mid}{2}\right\rbrace\\
 	&\geq&\frac{1}{2}\left\lbrace\frac{ \|(\Re(A)+\Im(A))^2+(\Re(A)-\Im(A))^2\|}{2}+\frac{\mid \|\Re(A)+\Im(A)\|^2-\|\Re(A)-\Im(A)\|^2\mid}{2}\right\rbrace\\ &=&\frac{\|A^*A+AA^*\|}{4}+\frac{\mid \|\Re(A)+\Im(A)\|^2-\|\Re(A)-\Im(A)\|^2\mid}{4}.
 \end{eqnarray*}
This completes the proof.
\end{proof}

	\begin{remark}
(i) Clearly, the inequality in Theorem \ref{thnn} refines the first inequality in (\ref{d}). \\
(ii)  If $ w^2(A)=\frac{\|A^*A+AA^*\|}{4}$, then $ \|\Re(A)+\Im(A)\|=\|\Re(A)-\Im(A)\|,$ but the converse is not necessarily true.\\
(ii) In \cite[Th. 2.9]{PK21}, Bhunia and Paul proved that if $ A\in\mathcal{B}(\mathcal{H})$, then  
\begin{eqnarray}\label{pk2}
	w^2(A)&\geq& \frac{1}{4} \left \|A^*A+AA^* \right\| +  \frac{1}{2}\mid \|\Re(A)\|^2-\|\Im(A)\|^2 \mid.
\end{eqnarray}
Considering the same examples as in Remark \ref{remark1}(ii), we conclude that the inequality in Theorem \ref{thnn} and the inequality (\ref{pk2}) are not comparable, in general.

\end{remark}

Next we obtain an upper bound for the numerical radius which improve on (\ref{e}). For this purpose, first we need the following lemma, based on polarization principle. 

\begin{lemma}\label{lem1}
	Let $ A \in \mathcal{B}(\mathcal{H})$, and let $ x,y \in \mathcal{H}$. Then 
	
	\begin{eqnarray*}
	\langle Ax,x \rangle &=&\frac{\langle A(x+y),x+y \rangle - \langle A(x-y),x-y \rangle}{4}\\
	&& +{\rm i} \frac{\langle A(x+{\rm i} y), x+{\rm i} y \rangle-\langle A(x-{\rm i} y),x-{\rm i} y\rangle}{4}.
	\end{eqnarray*}
	
\end{lemma}

\begin{theorem}
If $ A\in\mathcal{B}(\mathcal{H})$, then $$w(A)\leq\frac{1}{2} \left (\|A\|^2+w^2(\tilde{A})+w(|A|\tilde{A}+\tilde{A}|A|)\right )^{\frac{1}{2}}.$$
\label{th1}\end{theorem}
\begin{proof}
	First we note that $$ w(A) =\sup_{\theta\in\mathbb{R}}\|\Re(e^{\rm i\theta}A)\|=\sup_{\theta\in\mathbb{R}} w(\Re(e^{\rm i\theta}A)).$$ 
	Let $ A=U|A|$ be the polar decomposition of $A$. Then, by lemma \ref{lem1} we have,
	\begin{eqnarray*}
&&\langle e^{\rm i\theta} Ax,x\rangle=\langle e^{\rm i \theta}|A|x,U^*x\rangle\\ 
& =& \frac{1}{4} \left (\langle |A|(e^{\rm i \theta}x+U^*x),e^{\rm i \theta}x+U^*x \rangle - \langle |A|(e^{\rm i \theta}x-U^*x),e^{\rm i \theta}x-U^*x \rangle \right)\\
&+& \frac{{\rm i}}{4} \left (\langle |A|(e^{{\rm i} \theta}x+{\rm i}  U^*x),e^{{\rm i} \theta}x+{\rm i}  U^*x\rangle-\langle |A|(e^{{\rm i} \theta}x-{\rm i}  U^*x),e^{{\rm i}\theta}x-{\rm i} U^*x\rangle \right ). 
	\end{eqnarray*}
 Therefore, we have   
 \begin{eqnarray*}
 	Re\langle e^{{\rm i}\theta} Ax,x\rangle &= & \frac{1}{4}(\langle |A|(e^{{\rm i}\theta}x+U^*x),e^{{\rm i}\theta}x+U^*x \rangle - \langle |A|(e^{{\rm i}\theta}x-U^*x),e^{{\rm i}\theta}x-U^*x \rangle)\\
 	&\leq& \frac{1}{4}\langle |A|(e^{{\rm i}\theta}x+U^*x),e^{{\rm i}\theta}x+U^*x \rangle \\ 
 	&=& \frac{1}{4}\langle (e^{-{\rm i}\theta}+U)|A|(e^{{\rm i}\theta}+U^*)x,x \rangle \\
 	 & \leq& \frac{1}{4} \|(e^{-{\rm i}\theta}+U) |A|(e^{{\rm i}\theta}+U^*)\|  \\ 
 	 & =& \frac{1}{4} \| |A|^\frac{1}{2}(e^{{\rm i}\theta}+U^*)(e^{-{\rm i}\theta}+U)|A|^\frac{1}{2}\|,\hspace{.5cm}  (\|A^*A\|=\|AA^*\|) \\ 
 	 & =& \frac{1}{4} \|2|A|+e^{{\rm i}\theta}\tilde{A}+e^{-{\rm i}\theta}{\tilde{A}}^*\|  \\ & =& \frac{1}{2}\||A|+\Re(e^{{\rm i}\theta}\tilde{A})\| \\ 
 	 & =&\frac{1}{2}\left \|\left (|A|+\Re(e^{{\rm i}\theta}\tilde{A}) \right)^2\right \|^\frac{1}{2} \\
 	  & =& \frac{1}{2} \left \||A|^2+(\Re  (e^{{\rm i}\theta}\tilde{A}))^2+|A| \Re ( e^{{\rm i}\theta}\tilde{A})+\Re ( e^{{\rm i}\theta}\tilde{A})|A|\right \|^\frac{1}{2} \\ 
 	  & = & \frac{1}{2} \left \||A|^2+(\Re  (e^{{\rm i}\theta}\tilde{A}))^2 + \Re(e^{{\rm i}\theta}(|A|\tilde{A}+\tilde{A}|A|))\right \|^\frac{1}{2} \\ &\leq&\frac{1}{2}\left(\|A\|^2+\|\Re  (e^{{\rm i}\theta}\tilde{A}) \|^2+\| \Re(e^{{\rm i}\theta}(|A|\tilde{A}+\tilde{A}|A|))\|\right)^\frac{1}{2} \\
 	   &\leq& \frac{1}{2}\left(\|A\|^2+w^2(\tilde{A})+ w(|A|\tilde{A}+\tilde{A}|A|)\right)^\frac{1}{2}.
    \end{eqnarray*}
Since $Re\langle e^{{\rm i}\theta} Ax,x\rangle=\langle \Re(e^{{\rm i}\theta} A)x,x\rangle$, so taking supremum over $\theta \in \mathbb{R}$, we get
$$w(A) \leq \frac{1}{2}\left(\|A\|^2+w^2(\tilde{A})+ w(|A|\tilde{A}+\tilde{A}|A|)\right)^{ \frac{1}{2}},$$ as desired.
\end{proof}
	
\begin{remark} \label{rem1}
From \cite{G} we have if $ A,X\in\mathcal{B}(\mathcal{H})$, then $w(A^*X+XA)\leq 2 \|A\|w(X) $ and so
\begin{eqnarray*}
\|A\|^2+w^2(\tilde{A})+w(|A|\tilde{A}+\tilde{A}|A|)  &\leq&\|A\|^2+w^2(\tilde{A})+2\||A|\|w(\tilde{A})\\
&=&\|A\|^2+w^2(\tilde{A})+2\|A\|w(\tilde{A}) \\ 
&=&(\|A\|+w(\tilde{A}))^2.
\end{eqnarray*}	
Therefore,  $$w(A)\leq\frac{1}{2}\left(\|A\|^2+w^2(\tilde{A})+ w(|A|\tilde{A}+\tilde{A}|A|)\right)^\frac{1}{2}\leq\frac{1}{2}(\|A\|+w(\tilde{A})).$$
Thus the inequality in Theorem \ref{th1} is stronger than the inequality (\ref{e}). 
\end{remark}

Next we need the following two lemmas, first one known as Heinz inequality and second one known as Buzano's inequality.

\begin{lemma}(\cite{J})\label{p1}
 Let $ A \in\mathcal{B}(\mathcal{H})$. Then, for all $x,y\in\mathcal{H}$, we have
  $$ |\langle Ax,y \rangle|^2\leq\langle |A|^{2\alpha}x,x \rangle \langle |A^*|^{2(1-\alpha)}y,y \rangle, $$ 
  for all $ \alpha \in [0,1]$.
\end{lemma}

\begin{lemma}(\cite{I})\label{p2}
	Let $a,b,e\in\mathcal{H}$ with $\|e\|=1$.
	 $$ |\langle a,e \rangle \langle e,b\rangle |  \leq\frac{1}{2}(|\langle a,b \rangle|+\|a\|\|b\|).$$
\end{lemma}

Now we are in a position to prove an upper bound which improve on both the upper bounds in  (\ref{d}) and (\ref{b}).

\begin{theorem}	\label{th5}
	If $ A\in\mathcal{B}(\mathcal{H})$, then 
	$$ w^2(A)\leq\frac{1}{4} \left \| |A|^{4\alpha}+|A^*|^{4(1-\alpha)} \right \|+\frac{1}{2} w \left( |A^*|^{2(1-\alpha)}|A|^{2\alpha}\right), $$ 
	for all $ \alpha \in [0,1]$.
\end{theorem}
	
	\begin{proof}
		Let $x\in \mathcal{H}$ with $\|x\|=1$, then by Lemma \ref{p1}, we have
		\begin{eqnarray*}
			|\langle Ax,x \rangle|^2 &\leq& \langle |A|^{2\alpha}x,x \rangle\langle |A^*|^{2(1-\alpha)}x,x \rangle\\
			&=& \langle |A|^{2\alpha}x,x \rangle\langle x, |A^*|^{2(1-\alpha)}x \rangle\\ 
			&\leq& \frac{1}{2} \| |A|^{2\alpha}x\| \||A^*|^{2(1-\alpha)}x\|+\frac{1}{2}|\langle |A|^{2\alpha}x,|A^*|^{2(1-\alpha)}x \rangle|,\,\,(\textit{by Lemma \ref{p2}}) \\ 
			&=& \frac{1}{2} \langle |A|^{4\alpha}x,x\rangle^\frac{1}{2} \langle |A^*|^{4(1-\alpha)}x,x \rangle^\frac{1}{2}+\frac{1}{2}|\langle |A|^{2\alpha}x,|A^*|^{2(1-\alpha)}x \rangle| \\
			 &\leq&  \frac{1}{4} \left( \langle |A|^{4\alpha}x,x\rangle + \langle |A^*|^{4(1-\alpha)}x,x \rangle\right)+\frac{1}{2}|\langle |A|^{2\alpha}x,|A^*|^{2(1-\alpha)}x \rangle| \\ 
			 &=& \frac{1}{4}  \langle \left( |A|^{4\alpha} +  |A^*|^{4(1-\alpha)}\right) x,x \rangle+\frac{1}{2}|\langle |A^*|^{2(1-\alpha)}|A|^{2\alpha}x,x \rangle| \\ 
			 &\leq& \frac{1}{4} \| |A|^{4\alpha} +  |A^*|^{4(1-\alpha)}\| +\frac{1}{2} w\left( |A^*|^{2(1-\alpha)}|A|^{2\alpha}\right).
		\end{eqnarray*}
		Therefore, taking supremum over $x\in \mathcal{H}$, $\|x\|=1$, we get $$ w^2(A)\leq\frac{1}{4}\| |A|^{4\alpha}+|A^*|^{4(1-\alpha)} \|+\frac{1}{2} w \left( |A^*|^{2(1-\alpha)}|A|^{2\alpha}\right), $$ as desired.
	\end{proof}
	
In particular, considering $\alpha=\frac{1}{2}$ in Theorem \ref{th5}, we get the following inequality \cite[Cor. 2.6]{K} 
\begin{eqnarray}\label{eqnp1}
w^2(A)\leq\frac{1}{4}\| |A|^{2}+|A^*|^{2} \|+\frac{1}{2} w \left( |A^*||A|\right).
\end{eqnarray} 

Next, considering the minimum over $\alpha \in [0,1]$ in Theorem \ref{th5}, we get the following corollary.

\begin{cor}\label{pcor1}
	If $ A\in\mathcal{B}(\mathcal{H})$, then $$ w^2(A)\leq\min_{0\leq\alpha\leq 1}\left\lbrace\frac{1}{4} \left \| |A|^{4\alpha}+|A^*|^{4(1-\alpha)}\right \|+\frac{1}{2} w \left( |A^*|^{2(1-\alpha)}|A|^{2\alpha}\right)\right\rbrace. $$
\end{cor}

Clearly, we have 
$$ \min_{0\leq\alpha\leq 1}\left\lbrace\frac{1}{4} \left \| |A|^{4\alpha}+|A^*|^{4(1-\alpha)} \right\|+\frac{1}{2} w \left( |A^*|^{2(1-\alpha)}|A|^{2\alpha}\right)\right\rbrace \leq \frac{1}{4}\| |A|^{2}+|A^*|^{2} \|+\frac{1}{2} w \left( |A^*||A|\right).$$
In order to appreciate the inequality in Corollary \ref{pcor1} we give the following example. 
\begin{example}
	Let $$ A=  \begin{pmatrix}
	0 & 1 & 0 \\
	0  & 0 & 2\\
	0  & 0  & 0\\
	\end{pmatrix}.$$
Then by simple calculation, we have	
	 $$ |A|^{4\alpha}+ |A^*|^{4(1-\alpha)}=\begin{pmatrix}
	1 & 0 & 0 \\
	0 & 1+16^{1-\alpha} & 0 \\
	0 & 0 &  16^{\alpha} \\
	\end{pmatrix} \,\, \textit{and} \,\,\,\, |A^*|^{2(1-\alpha)} |A|^{2\alpha}=\begin{pmatrix}
	1 & 0 & 0 \\
	0 & 4^{1-\alpha} & 0 \\
	0 & 0 &  0 \\
	\end{pmatrix}.$$
Therefore, we have	
	\begin{eqnarray*}
	 \| |A|^{4\alpha}+ |A^*|^{4(1-\alpha)}\|&=& \begin{cases}
		1+16^{1-\alpha} \hspace{.5cm} \textit{if}\,\, \,\, 0\leq\alpha\leq r_0 \\
		16^{\alpha}  \hspace{1cm} \textit{if}\,\, \,\, r_0\leq\alpha\leq1,
	\end{cases}\end{eqnarray*}
where $16^{r_0}=\frac{1+\sqrt{65}}{2}$. Also,
\begin{eqnarray*}
w(|A^*|^{2(1-\alpha)} |A|^{2\alpha})=\begin{cases}
	4^{1-\alpha}\hspace{.5cm} \textit{if}\,\, \,\, 0\leq\alpha<1\\
	1 \hspace{1cm} \textit{if}\,\, \,\, \alpha=1.
\end{cases}
\end{eqnarray*}
Now, 
\begin{eqnarray*}	
	\min_{0\leq\alpha\leq r_0} \left\lbrace\frac{1+16^{1-\alpha}}{4}+\frac{4^{1-\alpha}}{2}\right\rbrace&=&\left\lbrace\frac{33+\sqrt{65}}{4(1+\sqrt{65})}+\frac{4\sqrt{2}}{2\sqrt{1+\sqrt{65}}}\right\rbrace\approx2.0724.\\
	\min_{r_0\leq\alpha\leq 1} \left\lbrace\frac{16^{\alpha}}{4}+\frac{4^{1-\alpha}}{2}\right\rbrace&=&\left\lbrace\frac{1+\sqrt{65}}{8}+\frac{4\sqrt{2}}{2\sqrt{1+\sqrt{65}}}\right\rbrace\approx2.0724.	\end{eqnarray*}
	 Thus, $$ \min_{0\leq\alpha\leq 1}\left\lbrace\frac{1}{4}\| |A|^{4\alpha}+|A^*|^{4(1-\alpha)} \|+\frac{1}{2} w \left( |A^*|^{2(1-\alpha)}|A|^{2\alpha}\right)\right\rbrace\approx2.0724. $$ 
Also, we have $$\frac{1}{4}\| |A|^{2}+|A^*|^{2} \|+\frac{1}{2} w \left( |A^*||A|\right)=\frac{9}{4}=2.25.$$	 
Hence, for the matrix $A$,
$$ \min_{0\leq\alpha\leq 1}\left\lbrace\frac{1}{4}\| |A|^{4\alpha}+|A^*|^{4(1-\alpha)} \|+\frac{1}{2} w \left( |A^*|^{2(1-\alpha)}|A|^{2\alpha}\right)\right\rbrace <\frac{1}{4}\| |A|^{2}+|A^*|^{2} \|+\frac{1}{2} w \left( |A^*||A|\right).$$ 
	\end{example}

Thus the inequality in Corollary \ref{pcor1} is a refinement of  the inequality (\ref{eqnp1}) and  hence it also refines inequalities  (\ref{d}) and (\ref{b}).

\section{Numerical radius inequalities of product of operators}

\noindent We begin this section with the following two lemmas, first one can be found in \cite[p. 20]{A} and second one is obvious.

\begin{lemma} 
 Let $A\in\mathcal{B}(\mathcal{H})$  be posittive and  let $x\in\mathcal{H}$ with $\|x\|=1$. Then for all $r\geq1$, we have 
 $$\langle Ax,x\rangle^r\leq\langle A^rx,x\rangle. $$
\label{lem4}\end{lemma}

\begin{lemma}
For all $ a,b\in\mathbb{R} $, we have $ |a+b|\leq\sqrt{2}|a+{\rm i} b|.$
\label{lem5}\end{lemma}

We now prove an improvement of (\ref{f}) .
 
\begin{theorem}\label{pth9}
Let $ A,B \in\mathcal{B}(\mathcal{H}$). Then, for all $r\geq1$, 
 $$ w^{2r}(B^*A)\leq\frac{1}{2} w^{2} \left (|A|^{2r}+{\rm i}|B|^{2r}\right ).$$ 
\end{theorem}

\begin{proof}
	Let $x \in \mathcal{H}$ with $\|x\|=1$. Then we have,
\begin{eqnarray*}
	|\langle B^*Ax,x\rangle|^{2r}&=& |\langle Ax,Bx\rangle|^{2r}\\
&\leq& \|Ax\|^{2r}\|Bx\|^{2r} \\ 
&\leq&  \langle|A|^{2r}x,x\rangle\langle|B|^{2r}x,x\rangle,\,\, (\textit{by Lemma \ref{lem4}}) \\
 &\leq& \frac{1}{4}\left(  \langle|A|^{2r}x,x\rangle+\langle|B|^{2r}x,x\rangle\right)^2 \\ &\leq& \frac{1}{2}\left|  \langle|A|^{2r}x,x\rangle+{\rm i}\langle|B|^{2r}x,x\rangle\right|^2,\,\, (\textit{by Lemma \ref{lem5}})  \\ &=&\frac{1}{2}\left|  \langle\left(|A|^{2r}+{\rm i}|B|^{2r}\right) x,x\rangle\right|^2\\ 
 &\leq& \frac{1}{2}w^2 \left(|A|^{2r}+{\rm i}|B|^{2r}\right).
 \end{eqnarray*}
Taking supremum over $x \in \mathbb{H}$, $\|x\|=1$, we get
 $$ w^{2r}(B^*A)\leq\frac{1}{2} w^2(|A|^{2r}+{\rm i }|B|^{2r}),$$ as required.
\end{proof}

It is easy to verify that, $w^2(|A|^{2r}+{\rm i }|B|^{2r})\leq \||A|^{4r}+|B|^{4r} \|.$
Thus, the inequality in Theorem \ref{pth9} refines  the inequality (\ref{f}) for $ r\geq2. $

For an improvement of (\ref{h}) we need the following lemma, that can be found in \cite{ar}.

\begin{lemma}\label{2a}
Let $f$ be a non-negative increasing convex function on $ [0,\infty)$ and let $ A,B\in\mathcal{B}(\mathcal{H}) $ be positive. Then 
$$ \left \| f\left(\frac{A+B}{2}\right) \right \|\leq  \left \| \frac{f(A)+f(B)}{2} \right \|.$$
\end{lemma} 

\begin{theorem}\label{th4}
If $ A,B \in\mathcal{B}(\mathcal{H})$, then for $r\geq1$,
 $$w^{2r}(B^*A)\leq\frac{1}{2}\left(\frac{ \left \||B|^{2}|A|^{2}+|A|^{2}|B|^{2} \right \|}{2}\right)^r+\frac{1}{4} \left \| |B|^{4r}+|A|^{4r}\right\|.$$
\end{theorem}

\begin{proof}
	Let $x \in \mathcal{H}$ with $\|x\|=1$. Then we have,	
\begin{eqnarray*}
	|\langle B^*Ax,x\rangle|^{2}&=& |\langle Ax,Bx\rangle|^{2}\\ 
	&\leq& \|Ax\|^{2}\|Bx\|^{2}\\
&\leq&  \langle|A|^{2}x,x\rangle\langle|B|^{2}x,x\rangle \\ &\leq& \frac{1}{4}\left(  \langle|A|^{2}x,x\rangle+\langle|B|^{2}x,x\rangle\right)^2 \\ &=& \frac{1}{4}\left(  \langle\left(|A|^{2}+|B|^{2}\right) x,x\rangle\right)^2 \\ &\leq&\frac{1}{4}  \langle\left(|A|^{2}+|B|^{2}\right)^2 x,x\rangle \\  &=& \frac{1}{4}  \langle\left(|A|^{4}+|B|^{4}+|A|^{2}|B|^{2}+|B|^{2}|A|^{2}\right) x,x\rangle \\  &\leq& \frac{1}{4}  \|\left(|A|^{4}+|B|^{4}+|A|^{2}|B|^{2}+|B|^{2}|A|^{2}\right)\| \\ &\leq&  \frac{1}{4}  \||A|^{4}+|B|^{4}\|+\frac{1}{4}\||A|^{2}|B|^{2}+|B|^{2}|A|^{2}\|. 
\end{eqnarray*}
Thus we have,
\begin{eqnarray*}
 	|\langle B^*Ax,x \rangle|^{2r}&\leq&  \left( \frac{1}{4}\||A|^{4}+|B|^{4}\|+\frac{1}{4}\||A|^{2}|B|^{2}+|B|^{2}|A|^{2}\| \right)^r \\ 
 	&\leq& \frac{1}{2}  \left( \frac{\| |A|^{4}+|B|^{4}\|}{2} \right)^r +\frac{1}{2}\left(\frac{\| |A|^{2}|B|^{2}+|B|^{2}|A|^{2}\|}{2}\right)^r \\ 
 	&\leq& \frac{1}{4}   \| |A|^{4r}+|B|^{4r}\| +\frac{1}{2}\left(\frac{\| |A|^{2}|B|^{2}+|B|^{2}|A|^{2}\|}{2}\right)^r.\, (\textit{by Lemma \ref{2a}})  \end{eqnarray*} 
  Hence, taking supremum over $x \in \mathcal{H}$, $\|x\|=1$, we get 
  $$ w^{2r}(B^*A)\leq\frac{1}{2}\left(\frac{\||B|^{2}|A|^{2}+|A|^{2}|B|^{2}\|}{2}\right)^r+\frac{1}{4}\| |B|^{4r}+|A|^{4r}\|.$$
\end{proof}

\begin{remark} By using convexity property of $f(t)=t^r, r\geq 1$, we have
\begin{eqnarray*} 
	&& \frac{1}{2}\left(\frac{\||B|^{2}|A|^{2}+|A|^{2}|B|^{2}\|}{2}\right)^r+\frac{1}{4}\| |B|^{4r}+|A|^{4r}\|\\
	 &=& \frac{1}{2}w^r\left(\frac{ |B|^{2}|A|^{2}+|A|^{2}|B|^{2}}{2}\right)+\frac{1}{4}\| |B|^{4r}+|A|^{4r}\| \\ &\leq&  \frac{1}{2}\left(\frac{1}{2} w(|B|^{2}|A|^{2})+\frac{1}{2} w(|A|^{2}|B|^{2})\right)^r+\frac{1}{4}\| |B|^{4r}+|A|^{4r}\| \\ &\leq& \frac{1}{4}\left( w^r(|B|^{2}|A|^{2})+ w^r(|A|^{2}|B|^{2})\right)+\frac{1}{4}\| |B|^{4r}+|A|^{4r}\| \\ &=&\frac{1}{2} w^r(|B|^{2}|A|^{2})+\frac{1}{4}\| |B|^{4r}+|A|^{4r}\|.
\end{eqnarray*}
Thus, for all $r\geq 1$ , we have 
\begin{eqnarray*}  w^{2r}(B^*A)&\leq&\frac{1}{2}\left(\frac{\||B|^{2}|A|^{2}+|A|^{2}|B|^{2}\|}{2}\right)^r+\frac{1}{4}\| |B|^{4r}+|A|^{4r}\| \\
	&\leq&\frac{1}{2} w^r(|B|^{2}|A|^{2})+\frac{1}{4}\| |B|^{4r}+|A|^{4r}\|.
\end{eqnarray*}
Therefore,  the inequality in Theorem \ref{th4} is a refinement of that in  (\ref{h}). 
\end{remark}

Finally, as a direct application of Theorem \ref{thnn}, we obtain the following inequality for the generalized commutators of operators. We use the idea used in the proof of  \cite[Th. 3.2]{bp}.

\begin{theorem}\label{pth1}
	Let $A,B,X,Y\in \mathcal{B}(\mathcal{H})$. Then 
	\begin{eqnarray*}
&&	w(AXB \pm BYA)\\
	&\leq& 2\sqrt{2}\|B\|\max  \left\{\|X\|,\|Y\| \right\}\sqrt{ w^2(A)-\frac{\mid \|\Re(A)+\Im(A)\|^2-\|\Re(A)-\Im(A)\|^2\mid}{4}   }.
	\end{eqnarray*}
\end{theorem}

\begin{proof}
	Let $x\in {\mathcal{H}}$ with $\|x\|=1$. If $\|X\|\leq 1$ and $\|Y\|\leq 1$,  then by Cauchy Schwarz inequality, we get 
	\begin{eqnarray*}
		|\langle (AX\pm YA)x,x\rangle|&\leq&  |\langle Xx,A^*x\rangle|+|\langle Ax,Y^*x\rangle| \\
		&\leq& \|A^*x\|+ \|Ax\| \\ 
		&\leq &  \sqrt{2}(\|A^*x\|^2+ \|Ax\|^2)^{\frac{1}{2}}\\ 
		&\leq&  \sqrt{2}\|AA^*+A^*A\|^{\frac{1}{2}}.
	\end{eqnarray*}
	Taking supremum over $\|x\|=1$, we get that, if $\|X\|\leq 1$ and $\|Y\|\leq 1$, then
	\begin{eqnarray}\label{eqnth1}
	w(AX\pm YA)&\leq& \sqrt{2}\|AA^*+A^*A\|^{\frac{1}{2}}.
	\end{eqnarray}
	Now we consider the general case, that is,  $X,Y\in \mathcal{B}(\mathcal{H})$ are arbitrary. 
	If $X=Y=0$, then (\ref{eqnth1}) holds trivially. If $\max  \left\{\|X\|,\|Y\| \right\}\neq 0$, then  $\left \| \frac{X}{\max  \left\{\|X\|,\|Y\| \right\}}\right\|\leq 1$ and $\left \| \frac{Y}{\max  \left\{\|X\|,\|Y\| \right\}}\right\|\leq 1$, and so it follows from the inequality (\ref{eqnth1}) that
	\begin{eqnarray}\label{eqnth2}
	w(AX\pm YA)\leq \sqrt{2}\max  \left\{\|X\|,\|Y\| \right\}\|AA^*+A^*A\|^{\frac{1}{2}}.
	\end{eqnarray}
	Replacing $X$ by $XB$ and $Y$ by $BY$ in the inequality (\ref{eqnth2}) we get,
	\begin{eqnarray*}
		w(AXB\pm BYA)\leq \sqrt{2}\max  \left\{\|XB\|,\|BY\| \right\}\|AA^*+A^*A\|^{\frac{1}{2}}.
	\end{eqnarray*}
	This implies that
	\begin{eqnarray*}
		w(AXB\pm BYA) \leq \sqrt{2} \|B\| \max  \left\{\|X\|,\|Y\| \right\}\|AA^*+A^*A\|^{\frac{1}{2}}.
	\end{eqnarray*}
Therefore, by using the inequality in Theorem \ref{thnn} we have,
\begin{eqnarray*}
	&&	w(AXB \pm BYA)\\
	&\leq& 2\sqrt{2}\|B\|\max  \left\{\|X\|,\|Y\| \right\}\sqrt{ w^2(A)-\frac{\mid \|\Re(A)+\Im(A)\|^2-\|\Re(A)-\Im(A)\|^2\mid}{4}   },
\end{eqnarray*}
as desired.
\end{proof}

Considering $X=Y=I$ (the identity operator) in Theorem \ref{pth1} we get the following corollary. 

\begin{cor}\label{corth1}
	If $A,B\in \mathcal{B}(\mathcal{H})$, then
	\begin{eqnarray*}
	w(AB\pm BA) &\leq & 2\sqrt{2} \|B\| \sqrt{ w^2(A)-\frac{\mid \|\Re(A)+\Im(A)\|^2-\|\Re(A)-\Im(A)\|^2\mid}{4} }.
	\end{eqnarray*}
\end{cor}

\begin{remark} Let $A,B\in \mathcal{B}(\mathcal{H})$.
	(i) Fong and Holbrook \cite{G} proved that 
	\begin{eqnarray}\label{Fong}
	w(AB+ BA) \leq  2\sqrt{2} \|B\| w(A).
	\end{eqnarray}
	Clearly, the inequality in Corollary \ref{corth1} is stronger than the inequality (\ref{Fong}).\\
	(ii) Hirzallah and Kittaneh \cite{HK} improved on the inequality (\ref{Fong}) to prove that
	\begin{eqnarray}\label{Hirzallah}
	w(AB \pm BA)&\leq& 2\sqrt{2}\|B\|\sqrt{ w^2(A)-\frac{|~~\|\Re (A)\|^2-\|\Im (A)\|^2~~|}{2} }.
	\end{eqnarray}	
	Considering the same examples as in Remark \ref{remark1}(ii), we see that the inequalities in Corollary \ref{corth1} and (\ref{Hirzallah}) are not comparable, in general.\\
	(iii) It follows from the Corollary \ref{corth1} that if $	w(AB\pm BA) =  2\sqrt{2} \|B\| w(A)$, then $$\|\Re(A)+\Im(A)\|=\|\Re(A)-\Im(A)\|.$$

\end{remark}

\noindent \textbf{Data availability statement.}\\
Data sharing is not applicable to this article as no new data were created or analyzed in this study.

\bibliographystyle{amsplain}

\end{document}